\documentclass[11pt]{article}

\usepackage{latexsym}
\usepackage{amssymb}
\usepackage{amsthm}
\usepackage{amsmath}

\newtheorem{theorem}{Theorem}[section]
\newtheorem{lemma}{Lemma}[section]
\newtheorem{prop}{Proposition}[section]
\newtheorem{corollary}{Corollary}[section]

\numberwithin{equation}{section}

\newcommand{\ZZ}{\mathbb{Z}}

\newcommand{\cP}{\mathcal{P}}

\newcommand{\tH}{\tilde{H}}

\newcommand{\wt}{\mathrm{wt}}
\newcommand{\tlambda}{\tilde{\lambda}}
\newcommand{\tmu}{\tilde{\mu}}
\newcommand{\tnu}{\tilde{\nu}}
\newcommand{\tkappa}{\tilde{\kappa}}

\makeatletter
\renewcommand{\@makefnmark}{\mbox{\textsuperscript{}}}
\makeatother

\allowdisplaybreaks[1,2,3,4]

\begin{document}

\bibliographystyle{amsplain}

\title{A product formula for multivariate Rogers-Szeg\"o polynomials}
\author{Stephen Cameron and C. Ryan Vinroot} 
\date{}

\maketitle

\begin{abstract} Let $H_n(t)$ denote the classical Rogers-Szeg\"o polynomial, and let $\tH_n(t_1, \ldots, t_l)$ denote the homogeneous Rogers-Szeg\"o polynomial in $l$ variables, with indeterminate $q$.  There is a classical product formula for $H_k(t)H_n(t)$ as a sum of Rogers-Szeg\"o polynomials with coefficients being polynomials in $q$.  We generalize this to a product formula for the multivariate homogeneous polynomials $\tH_n(t_1, \ldots, t_l)$.  The coefficients given in the product formula are polynomials in $q$ which are defined recursively, and we find closed formulas for several interesting cases.  We then reinterpret the product formula in terms of symmetric function theory, where these coefficients become structure constants.
\end{abstract}

\section{Introduction}

For any $q \neq 1$, and any positive integer $n$, we let $(q)_n = (1-q) (1-q^2) \cdots (1-q^n)$, $(q)_0 = 1$, and for any non-negative integers $n$ and $r$ with $n \geq r$, we denote the standard $q$-binomial coefficient by $\binom{n}{k}_q = \frac{(q)_n}{(q)_r (q)_{n-r}}$.

The \emph{Rogers-Szeg\"o polynomial} in a single variable, denote $H_n(t)$, is defined to be
$$ H_n(t) = \sum_{r=0}^n \binom{n}{r}_q t^r.$$
The Rogers-Szeg\"o polynomials appeared in the proof of the famous Rogers-Ramanujan identities, in papers of Rogers \cite{Rog1, Rog2}, and were later studied by Szeg\"o as orthogonal polynomials \cite{Sz26}.  One of the key properties which the Rogers-Szeg\"o polynomials satisfy is the following identity giving a way to write the product of two Rogers-Szeg\"o polynomials as a sum of others (see \cite[Example 3.6]{An76}):

\begin{equation} \label{OrigProd}
H_k(t) H_n(t) = \sum_{r=0}^k \binom{k}{r}_q \binom{n}{r}_q (q)_r t^r H_{k+n-2r}(t).
\end{equation}

For any non-negative integers $r_1, \ldots, r_l$, such that $r_1 + \cdots + r_l = n$, define the $q$-multinomial coefficient to be $\binom{n}{r_1, \ldots, r_l}_q = \frac{(q)_n}{(q)_{r_1} \cdots (q)_{r_l}}$.  We then may define a homogeneous multivariate version of the Rogers-Szeg\"o polynomials in $l$ variables, which we define to be
$$ \tH_n(t_1, \ldots, t_l) = \sum_{r_1 + \cdots + r_l = n} \binom{n}{r_1, \ldots, r_l}_q t_1^{r_1} t_2^{r_2} \cdots t_l^{r_l}.$$

In particular, note that $\tH_n(t,1) = H_n(t)$.  The non-homogeneous version of the multivariate Rogers-Szeg\"o polynomial, $H(t_1, \ldots, t_{l-1}) = \tH(t_1, \ldots, t_{l-1}, 1)$, is considered in the book of G. Andrews \cite[Example 3.17]{An76}, while the homogeneous version was initially defined by Rogers \cite{Rog1, Rog2} in terms of their generating function, and some of their basic properties are given in the monograph of N. Fine \cite{Fi88}.  These have also been studied by K. Hikami in the context of mathematical physics \cite{Hi95, Hi97}.

The main result of this paper is a generalization of the product formula (\ref{OrigProd}) to the case of homogeneous multivariate Rogers-Szeg\"o polynomials.  In order to state this result precisely, we need a bit more notation.  First, let $\vec{m} \in \ZZ_{\geq 0}^{l-1}$ be a vector with $l-1$ non-negative integer entries, whose coordinates we write as $\vec{m} = [m_2, m_3, \ldots, m_l]$ (so that the $i$th coordinate is labeled $m_{i+1}$), and let $|\vec{m}| = \sum_{i=2}^l m_i$, and define $\wt(\vec{m}) = \sum_{i=2}^l i m_i$.  For any $i = 1, \ldots, l-1$, let $\vec{u}_i$ denote the unit vector with $1$ in the $i$th coordinate and $0$ elsewhere.  Now, given any $n, k \geq 0$, and any $\vec{m} \in \ZZ_{\geq 0}^{l-1}$, we define polynomials in $q$, $\theta_{\vec{m}, k,n}(q)$, recursively as follows.  If $|\vec{m}| > \mathrm{min} \{ k, n \}$, or $\wt(\vec{m}) > k + n$, define $\theta_{\vec{m}, k, n} = 0$, and define $\theta_{\vec{0}, k, n} = 1$ for any $k,n \geq 0$.  The recursive definition for $\theta_{\vec{m}, k, n} = \theta_{\vec{m}, k, n}(q)$ is then given by
\begin{equation} \label{thetadef}
\theta_{\vec{m}, k+1, n} = \theta_{\vec{m}, k, n} + \sum_{j=1}^{l-1} \binom{n+k - \wt(\vec{m}) + j + 1}{j}_q (q)_j \theta_{\vec{m} - \vec{u}_j,k,n} - \sum_{j=1}^{l-1} \binom{k}{j}_q (q)_j \theta_{\vec{m} - \vec{u}_j, k-j, n},
\end{equation} 
where we take $\theta_{\vec{a}, b, c} = 0$ if any of $b$, $c$, or any coordinate of $\vec{a}$ is negative.

Define $e_i(t_1, \ldots, t_l)$ to be the $i$th elementary symmetric polynomial.  The main result of this paper is Theorem \ref{ProdMain}, which is a product formula generalizing (\ref{OrigProd}), may be stated as follows:
\begin{align*}
\tH_k(t_1, \ldots, t_l) & \tH_n(t_1, \ldots, t_l) \\
                        & = \sum_{\vec{m} \in \ZZ_{\geq 0}^{l-1}} (-1)^{\wt(\vec{m})} \theta_{\vec{m}, k, n}(q) \left(\prod_{i=2}^{l} e_i(t_1, \ldots, t_l)^{m_i} \right) \tH_{k+n - \wt(\vec{m})}(t_1, \ldots, t_l).
\end{align*}

We prove the above formula in Section \ref{ProdSec}, and we also give several properties of the polynomials $\theta_{\vec{m}, k, n}(q)$ there.  In particular, we give several cases of closed formulas for $\theta_{\vec{m}, k, n}(q)$, including the case that $\theta_{r\vec{u}_1, k, n}(q) = \binom{k}{r}_q \binom{n}{r}_q (q)_r$ in Proposition \ref{ru1}, explaining how the generalized product formula implies (\ref{OrigProd}).  In Section \ref{SymmSec}, we reinterpret our results in terms of symmetric functions, giving an interpretation of the polynomials $\theta_{\vec{m}, k, n}(q)$ as structure constants with respect to a linear basis for the graded algebra of symmetric functions over $\ZZ[q]$.
\\
\\
\noindent{\bf Acknowledgments. }  The second-named author was supported by NSF grant DMS-0854849.

\section{The Product Formula} \label{ProdSec}

We begin with a recursion for the homogeneous Rogers-Szeg\"o polynomials $\tH_n(t_1, \ldots, t_l)$, due to Hikami \cite{Hi95, Hi97}.

\begin{prop} \label{RSRec}
Take $\tH_j(t_1, \ldots, t_l) = 0$ for $j <0$, and we have $\tH_0(t_1, \ldots, t_l) = 1$.  For any $n$, we have
$$\tH_{n+1}(t_1, \ldots, t_l) = \sum_{j=0}^{l-1} (-1)^j e_{j+1} (t_1, \ldots, t_l) \binom{n}{j}_q (q)_j \tH_{n-j}(t_1, \ldots, t_l).$$
\end{prop}

Using the recursion in Proposition \ref{RSRec}, and the recursive definition (\ref{thetadef}) of the polynomials $\theta_{\vec{m}, n, k}(q)$, we may prove our main result.

\begin{theorem} \label{ProdMain}  For any $k, n \geq 0$, we have
\begin{align*}
\tH_k(t_1, \ldots, t_l) & \tH_n(t_1, \ldots, t_l) \\
                        & = \sum_{\vec{m} \in \ZZ_{\geq 0}^{l-1}} (-1)^{\wt(\vec{m})} \theta_{\vec{m}, k, n}(q) \left(\prod_{i=2}^{l} e_i(t_1, \ldots, t_l)^{m_i} \right) \tH_{k+n - \wt(\vec{m})}(t_1, \ldots, t_l).
\end{align*}
\end{theorem}
\begin{proof} 
To simplify notation, we will suppress the variables $t_1, \ldots, t_l$, so that $e_j = e_j(t_1, \ldots, t_l)$ and $\tH_j = \tH_j(t_1, \ldots, t_l)$.  
Let $n \geq 0$.  Then $\tH_0 \tH_n = \tH_n$.  Since $\theta_{\vec{m},0,n} = 0$ whenever $\vec{m} \neq \vec{0}$, and $\theta_{\vec{0}, 0, 0} = 1$, the statement holds for $k=0$.  Now fix $k \geq 0$, and assume the statement holds for all indices $i \leq k$, so holds for products $\tH_i \tH_n$ for $0 \leq i \leq k$ and all $n \geq 0$.  Consider the product $\tH_{k+1} \tH_n$.  By the recursion in Proposition \ref{RSRec}, we have
$$\tH_{k+1} \tH_n = \left( \sum_{j=0}^{l-1} (-1)^j e_{j+1} \binom{k}{j}_q (q)_j \tH_{k-j} \right) \tH_n.$$
We may apply the induction hypothesis to each of the products $\tH_{k-j} \tH_n$, so that we have
\begin{align} \label{firsteqn}
\tH_{k+1}  \tH_n & = \sum_{j=0}^{l-1} (-1)^j e_{j+1} \binom{k}{j}_q (q)_j \left(\sum_{\vec{m} \in \ZZ_{\geq 0}^{l-1}} (-1)^{\wt(\vec{m})} \theta_{\vec{m}, k-j, n} \left(\prod_{i=2}^{l} e_i^{m_i} \right) \tH_{k-j+n - \wt(\vec{m})} \right) \notag \\
 & = \sum_{\vec{m} \in \ZZ_{\geq 0}^{l-1}}(-1)^{\wt(\vec{m})} \theta_{\vec{m}, k, n} \left(\prod_{i=2}^{l} e_i^{m_i} \right) e_1 \tH_{k+n-\wt(\vec{m})} \notag \\
 & \;\; + \sum_{j=1}^{l-1} \sum_{\vec{m} \in \ZZ_{\geq 0}^{l-1}} (-1)^{j + \wt(\vec{m})}\theta_{\vec{m}, k-j, n} \binom{k}{j}_q (q)_j e_{j+1}^{m_{j+1} + 1} \left( \prod_{2 \leq i \leq l \atop{i \neq j+1}} e_i^{m_i} \right) \tH_{k-j+n - \wt(\vec{m})}.
\end{align}
From Proposition \ref{RSRec}, we have 
$$e_1 \tH_{k+n - \wt(\vec{m})} = \tH_{k+n + 1 - \wt(\vec{m})} + \sum_{j=1}^{l-1} (-1)^{j+1} e_{j+1} \binom{k+n - \wt(\vec{m})}{j}_q (q)_j \tH_{k+n-j-\wt(\vec{m})}.$$
Subbing this into (\ref{firsteqn}), we obtain
\begin{align*} 
\tH&_{k+1} \tH_n = \\
& \sum_{\vec{m} \in \ZZ_{\geq 0}^{l-1}}(-1)^{\wt(\vec{m})} \theta_{\vec{m}, k, n} \left(\prod_{i=2}^l e_i^{m_i} \right) \tH_{k+n+1 - \wt(\vec{m})} \\ 
& \quad + \sum_{j=1}^{l-1} \sum_{\vec{m} \in \ZZ_{\geq 0}^{l-1}} (-1)^{j + 1 + \wt(\vec{m})} \theta_{\vec{m}, k, n} \binom{k+n - \wt(\vec{m})}{j}_q (q)_j e_{j+1}^{m_{j+1} + 1} \left( \prod_{2 \leq i \leq l \atop{i \neq j+1}} e_i^{m_i} \right) \tH_{k+n-j-\wt(\vec{m})}  \\
& \quad + \sum_{j=1}^{l-1} \sum_{\vec{m} \in \ZZ_{\geq 0}^{l-1}} (-1)^{j + \wt(\vec{m})}\theta_{\vec{m}, k-j, n} \binom{k}{j}_q (q)_j e_{j+1}^{m_{j+1} + 1} \left( \prod_{2 \leq i \leq l \atop{i \neq j+1}} e_i^{m_i} \right) \tH_{k-j+n - \wt(\vec{m})}.
\end{align*}
In the second and third sums above, we shift the index by replacing $\vec{m}$ with $\vec{m} - \vec{u}_j$.  Since we define $\theta_{\vec{a}, b, c} = 0$ if any coordinate of $\vec{a}$ is negative, this does not alter the terms which occur in the sum.  Note also that $\wt(\vec{m} - \vec{u}_j) = \wt(\vec{m}) -j -1$.  So, after this re-indexing, we have
\begin{align*} 
\tH&_{k+1} \tH_n = \sum_{\vec{m} \in \ZZ_{\geq 0}^{l-1}}(-1)^{\wt(\vec{m})} \theta_{\vec{m}, k, n} \left(\prod_{i=2}^l e_i^{m_i} \right) \tH_{k+n+1 - \wt(\vec{m})} \\ 
& + \sum_{j=1}^{l-1} \sum_{\vec{m} \in \ZZ_{\geq 0}^{l-1}} (-1)^{\wt(\vec{m})} \theta_{\vec{m}-\vec{u}_j, k, n} \binom{k+n - \wt(\vec{m})+j + 1}{j}_q (q)_j \left( \prod_{i=2}^l e_i^{m_i} \right) \tH_{k+n+1-\wt(\vec{m})}  \\
& - \sum_{j=1}^{l-1} \sum_{\vec{m} \in \ZZ_{\geq 0}^{l-1}} (-1)^{\wt(\vec{m})}\theta_{\vec{m} - \vec{u}_j, k-j, n} \binom{k}{j}_q (q)_j \left( \prod_{i=2}^l e_i^{m_i} \right) \tH_{k+n+1 - \wt(\vec{m})}.
\end{align*}
By the recursive definition (\ref{thetadef}) of $\theta_{\vec{m}, k, n}$, we may write the above as
$$\tH_{k+1} \tH_n = \sum_{\vec{m} \in \ZZ_{\geq 0}^{l-1}} (-1)^{\wt(\vec{m})} \theta_{\vec{m}, k+1, n} \left(\prod_{i=2}^l e_i^{m_i} \right) \tH_{k+1+n - \wt(\vec{m})}, $$
completing the induction.
\end{proof}

Although we do not have a closed formula for the polynomials $\theta_{\vec{m}, k, n}(q)$ in general, we do have closed expressions for several interesting cases.

\begin{prop} \label{uj} For any $j \geq 1$, $n, k \geq 0$, we have
$$ \theta_{\vec{u}_j, k, n}(q) = (q)_j \sum_{i = 0}^{k-1} \left[\binom{n+i}{j}_q - \binom{i}{j}_q \right].$$
\end{prop}
\begin{proof} The statement holds whenever $k=0$, since then $1 = |\vec{u}_j| > \mathrm{min} \{k,n\} = 0$, so $\theta_{\vec{u}_j, 0, n} = 0$ by definition.  Suppose the statement holds for $k$, for any $j \geq 1$, $n \geq 0$.  From the recursive definition (\ref{thetadef}), and from $\wt(\vec{u}_j) = j+1$, we then have
\begin{align*}
\theta_{\vec{u}_j, k+1, n} & = \theta_{\vec{u}_j, k, n} + (q)_j \binom{n+k}{j}_q - (q)_j \binom{k}{j}_q \\
& = (q)_j \sum_{i=0}^k \left[\binom{n+i}{j}_q - \binom{i}{j}_q\right],
\end{align*}
which completes the proof.
\end{proof}

The following is a special case with a slightly more involved argument.

\begin{prop} \label{|m|=k} Suppose $\vec{m} \in \ZZ^{l-1}_{\geq 0}$, $k, n \geq 0$, such that $|\vec{m}| = m_2 + \cdots + m_l = k$.  Then 
$$ \theta_{\vec{m}, k, n}(q) = \binom{k}{m_2, m_3, \ldots, m_l} \binom{n}{\wt(\vec{m}) - k}_q (q)_{\wt(\vec{m}) - k}.$$
\end{prop}
\begin{proof} First, denote the standard multinomial coefficient $\binom{k}{r_2, r_3, \ldots, r_l}$ by $\binom{k}{\vec{r}}$, where $\vec{r} = [r_2, r_3, \ldots, r_l]$ and $|\vec{r}| = r_2 + r_3 + \cdots + r_l = k$.  If any $r_j < 0$, define $\binom{k}{\vec{r}} = 0$.

The proof is by induction on $|\vec{m}|$, where if $|\vec{m}| = 0$, then $\vec{m} = \vec{0}$, and $\theta_{\vec{0}, k, n} = 1$ for any $k, n \geq 0$ by definition.  In particular, when $k = |\vec{0}| = 0$, the claimed formula also yields $1$.  

Assuming the formula holds when $|\vec{m}| = k$, suppose that $|\vec{m}| = k+1$, and recall from definition that if $|\vec{m}| > \mathrm{min} \{ h,n \}$, then $\theta_{\vec{m}, h, n}(q) = 0$.  In particular, if $|\vec{m}| = k+1$, then $\theta_{\vec{m}, k, n} = 0$, and $\theta_{\vec{m} - \vec{u}_j, k-j, n} = 0$ for any $j = 1, \ldots, l-1$.  So, applying the recursion (\ref{thetadef}), we obtain
$$ \theta_{\vec{m}, k+1, n}(q) = \sum_{j=1}^{l-1} \binom{n + k - \wt(\vec{m}) + j + 1}{j}_q (q)_j \theta_{\vec{m} - \vec{u}_j, k, n},$$
where, if $m_{j+1} \neq 0$, then $|\vec{m} - \vec{u}_j| = k$, and otherwise $\theta_{\vec{m} - \vec{u}_j, k, n} = 0$, and $\binom{k}{\vec{m}-\vec{u}_j} = 0$.  So, applying this observation and the induction hypothesis, we obtain
$$ \theta_{\vec{m}, k+1, n}(q) = \sum_{j=1}^{l-1} \binom{n+k - \wt(\vec{m}) + j+ 1}{j}_q (q)_j \binom{k}{\vec{m} - \vec{u}_j} \binom{n}{\wt(\vec{m} - \vec{u}_j) - k}_q (q)_{\wt(\vec{m} - \vec{u}_j) - k}.$$
Noting that $\wt(\vec{m} - \vec{u}_j) = \wt(\vec{m}) - j - 1$, we have
\begin{align*}
\binom{n + k - \wt(\vec{m}) + j + 1}{j}_q & \binom{n}{\wt(\vec{m}) - j - k -1}_q (q)_j (q)_{\wt(\vec{m}) - j - k - 1} \\
 & = \binom{n}{\wt(\vec{m}) - k - 1}_q (q)_{\wt(\vec{m}) - k - 1}.
\end{align*}
So,
$$\theta_{\vec{m}, k+1, n}(q) = \binom{n}{\wt(\vec{m}) - k - 1}_q (q)_{\wt(\vec{m}) - k - 1} \sum_{j=1}^{l-1} \binom{k}{\vec{m} - \vec{u}_j}.$$
Recalling the Pascal recursion for multinomial coefficients, $\binom{k+1}{\vec{m}} = \sum_{j=1}^{l-1} \binom{k}{\vec{m} - \vec{u}_j}$, completes the induction argument.
\end{proof}

Finally, we have the following, which explains how Theorem \ref{ProdMain} implies the classical product formula (\ref{OrigProd}).

\begin{prop} \label{ru1} For any $k, n, r \geq 0$, we have
$$ \theta_{r\vec{u}_1, k, n}(q) = \binom{k}{r}_q \binom{n}{r}_q (q)_r.$$
\end{prop}
\begin{proof} By definition, the statement holds if $r=0$, or for any $r > 1$,$ n \geq 0$ if $k =0$, since then $r = |r \vec{u}_1| > \mathrm{min} \{k,n\} = 0$, and $\binom{k}{r}_q = 0$ by definition.  Assuming the statement holds for $k$, then from (\ref{thetadef}), we have
\begin{align*}
\theta_{r\vec{u}_1, k+1, n} & = \theta_{r \vec{u}_1, k, n} + (q)_1 \binom{n+k-2r+2}{1}_q \theta_{(r-1)\vec{u}_1, k, n} - (q)_1 \binom{k}{1}_q \theta_{(r-1)\vec{u}_1, k-1, n} \\
 & = \binom{k}{r}_q \binom{n}{r}_q (q)_r  + (1 - q^{n+k-2r+2}) \binom{n}{r-1}_q \binom{k}{r-1}_q (q)_{r-1} \\
 & \quad \quad \quad - (1 - q^k) \binom{n}{r-1}_q \binom{k-1}{r-1}_q (q)_{r-1} \\
 & = \binom{n}{r}_q (q)_r \left( \binom{k}{r}_q + \frac{1 - q^{n+k-2r +2}}{1 - q^{n-r+1}} \binom{k}{r-1}_q - \frac{1 - q^k}{1 - q^{n-r+1}} \binom{k-1}{r-1}_q \right) \\
& = \binom{k+1}{r}_q \binom{n}{r}_1 (q)_r \left( \frac{1 - q^{k+1-r}}{1 - q^{k+1}} + \frac{(1-q^r)(1 - q^{n+k-2r+2})}{(1 - q^{k+1})(1 - q^{n-r+1})} \right. \\
& \quad \quad \quad \left. -\frac{(1-q^r)(1 - q^{k+1 - r})}{(1 - q^{k+1})(1 - q^{n-r+1})} \right) \\
& = \binom{k+1}{r}_q \binom{n}{r}_q (q)_r,
\end{align*}
where the last step is a direct computation.  This completes the induction.
\end{proof}

We would like other properties of the polynomials $\theta_{\vec{m}, k, n}(q)$ which further characterize them.  For example, since $\tH_k \tH_n = \tH_n \tH_k$, then one might expect that $\theta_{\vec{m}, k, n} = \theta_{\vec{m}, n, k}$.  In fact, with a somewhat tedious proof, one may obtain this fact directly from the recursive definition (\ref{thetadef}).  However, we prove this statement another way in the next section by a reinterpretation of Theorem \ref{ProdMain} in terms of bases of symmetric functions.

\section{Rogers-Szeg\"o Symmetric Functions} \label{SymmSec}

Recall that a symmetric polynomial $f \in \ZZ[t_1, \ldots, t_n]$ is a polynomial which is invariant under the action of the symmetric group $S_n$ permuting the variables.  Let $\Lambda_n$ denote the ring of symmetric polynomials in $\ZZ[t_1, \ldots, t_n]$, so $\Lambda_n = \ZZ[t_1, \ldots, t_n]^{S_n}$, and $\Lambda_n$ is also a $\ZZ$-module.

As in \cite[I.2]{Ma95}, let $\Lambda_n^k$ denote the submodule of $\Lambda_n$ consisting of homogeneous symmetric polynomials of degree $k$.  For any $m > n$, we may map an element $f(t_1, \ldots, t_m) \in \Lambda_n^k$ by sending $t_{n+1}, \ldots, t_m$ to $0$, which gives a system of projective maps 
$$p_{m,n}^k: \Lambda_m^k \rightarrow \Lambda_n^k,$$
from which we may form the inverse limit
$$ \Lambda^k = \lim_{\longleftarrow} \Lambda_n^k.$$
We then define the ring $\Lambda$ of symmetric functions over $\ZZ$ in the countably infinite set of variables $T = \{t_1, t_2, \ldots \}$ to be the direct sum
$$ \Lambda = \bigoplus_k \Lambda^k,$$
which is then a graded $\ZZ$-algebra.

Given an indeterminate $q$, we may define $\Lambda[q]$ by either the tensor product
$$\Lambda[q] = \Lambda \otimes_{\ZZ} \ZZ[q],$$
or, if we define $\Lambda[q]_m^k =\displaystyle \Lambda_m^k \otimes_{\ZZ} \ZZ[q]$, and extend the projective system as above to define $\Lambda[q]^k = \displaystyle \lim_{\longleftarrow} \Lambda[q]_m^k$, we can then equivalently define 
$$\Lambda[q] = \bigoplus_k \Lambda[q]^k.$$

We now note that the homogeneous Rogers-Szeg\"o polynomial indeed satisfies
$$ \tH_n(t_1, \ldots, t_l, 0, 0, \ldots, 0) = \tH_n(t_1, \ldots, t_l),$$
where the polynomial on the left has any number of variables more than $l$.  So, the homogeneous Rogers-Szeg\"o polynomial has an image in the graded $\ZZ[q]$-algebra $\Lambda[q]$ as described above.  We denote this image as $\tH_n(T) = \tH_n$, and call it the \emph{Rogers-Szeg\"o symmetric function}.

There is a large number of linear bases of $\Lambda$ as a $\ZZ$-module (or $\Lambda[q]$ as a $\ZZ[q]$-module) which are of interest.  In general, such bases are parameterized by the set $\cP$ of partitions of non-negative integers, where each $\Lambda^k$ has a basis of partitions of size $k$.  If $\lambda \in \cP$, we denote $\lambda$ as either $\lambda = (\lambda_1, \lambda_2, \ldots)$, where $\lambda_i \geq \lambda_{i+1}$ and $\sum_i \lambda_i = |\lambda|$, or as $\lambda = (1^{m_1} 2^{m_2} \cdots )$, where $m_j = m_j(\lambda)$ is the multiplicity of $j$ in $\lambda$, so $\sum_j j m_j = |\lambda|$

One important basis of $\Lambda$ (or $\Lambda[q]$) is given by the set of elementary symmetric functions.  In particular, for a positive integer $j$, define $e_j$ to be the symmetric function which is the projective limit of the elementary symmetric polynomial $e_j(t_1, \ldots, t_l)$ introduced in previous sections.  If $\lambda \in \cP$, with $\lambda = (\lambda_1, \lambda_2, \ldots) = (1^{m_1} 2^{m_2} \cdots)$, then define $e_{\lambda}$ as
$$ e_{\lambda} = e_{\lambda_1} e_{\lambda_2} \cdots = e_1^{m_1} e_2^{m_2} \cdots.$$
Then $\{ e_{\lambda} \, \mid \, |\lambda| = k \}$ is a $\ZZ$-basis (or a $\ZZ[q]$-basis) for $\Lambda^k$ (or $\Lambda[q]^k$), and $\{e_{\lambda} \, \mid \, \lambda \in \cP \}$ is a  $\ZZ$-basis (or a $\ZZ[q]$-basis) for $\Lambda$ (or $\Lambda[q]$) \cite[I.2]{Ma95}.

We now consider then fact that $\tH_n \in \Lambda[q]$.  Note that Proposition \ref{RSRec} turns into the following in $\Lambda[q]$:
\begin{equation} \label{RecSym}
\tH_{n+1} = \sum_{j=0}^n (-1)^j e_{j+1} \binom{n}{j}_q (q)_j \tH_{n-j}.
\end{equation}

We need the following lemma.

\begin{lemma} \label{HeCoeff}  If we expand $\tH_n \in \Lambda[q]$ in the elementary symmetric function basis over $\ZZ[q]$, the coefficient of $e_{(1^n)}$ is $1$.
\end{lemma}
\begin{proof} We prove by induction on $n$, and for $n=0$ we have $\tH_0 = 1 = e_{(0)}$ by definition, and for $n=1$, $\tH_1 = e_{(1)}$.  We assume the statement holds for all $j \leq n$, and write
\begin{equation} \label{Hexp}
\tH_j = \sum_{|\lambda| = j} c_{\lambda, j}(q) e_{\lambda} = \sum_{|\lambda| = j} c_{\lambda, j}(q) e_1^{m_1(\lambda)} e_2^{m_2(\lambda)} \cdots,
\end{equation}
where each $c_{\lambda, j}(q) \in \ZZ[q]$, and $c_{(1^j), j}(q) = 1$ for each $j \leq n$.  Now consider $\tH_{n+1}$, and using (\ref{RecSym}), we have
\begin{align*}
\tH_{n+1} & = \sum_{j=0}^n (-1)^j e_{j+1} \binom{n}{j}_q (q)_j \tH_{n-j}\\
         & = e_1 \tH_n + \sum_{j=1}^n (-1)^j e_{j+1} \binom{n}{j}_q (q)_j \tH_{n-j} \\
 & = e_1 \sum_{|\mu| = n} c_{\mu, n}(q) e_{\mu} + \sum_{j=1}^{n} \sum_{|\lambda| = j} c_{\lambda, j}(q) (-1)^j \binom{n}{j}_q (q)_j e_{j+1} e_{\lambda}.
\end{align*}
In the first sum, we have $c_{(1^n), n}(q) = 1$, so that the coefficient of $e_{(1^{n+1})}$ from the first term has coefficient 1.  The fact that $e_{(1^{n+1})}$ does not appear anywhere in the double sum completes the argument.
\end{proof}

For any partition $\lambda = (1^{m_1} 2^{m_2} 3^{m_3} \cdots)$, define $\tlambda = (2^{m_2} 3^{m_3} \cdots)$.  That is, $m_1(\tlambda) = 0$, while $m_j(\tlambda) = m_j(\lambda)$ when $j \geq 2$ (note that if $\lambda = (1^m)$, then $\tlambda$ is the empty partition).  Now, for any partition $\lambda$, we may consider the symmetric function $\tH_{m_1} e_{\tlambda}$, where $m_1 = m_1(\lambda)$, and $\tH_{m_1} e_{\tlambda} \in \Lambda[q]^k$ if $|\lambda| = k$.  Denote this symmetric function by $R_{\lambda}$.  

We may now reconsider how Theorem \ref{ProdMain} translates in the language of $\Lambda[q]$.  In terms of the product $\tH_k \tH_n$, given any $\vec{m}$ with a finite number of positive integer entries, we may think of $\vec{m} \in \ZZ_{\geq 0}^{l-1}$ with $\vec{m} = [m_2, m_3, \ldots, m_l]$.  If $\wt(\vec{m}) \leq k+n$, then $\vec{m}$ corresponds to a partition $\lambda$ such that $|\lambda| = k+n$ and $\tlambda = (2^{m_2} 3^{m_3} \cdots )$, and conversely, any $\vec{m}$ satisfying $\wt(\vec{m}) \leq k+n$ corresponds to a unique partition $\lambda$ of $k+n$.  Now, given any partition $\lambda$ of $k+n$, we define $\theta_{\lambda, k, n}(q)$ as
$$\theta_{\lambda, k, n}(q) = \theta_{\vec{m}, k, n}(q),$$
where $\vec{m} = [m_2(\lambda), m_3(\lambda), \ldots]$.  Since $\wt(\vec{m}) = |\tlambda|$ in this correspondence, we may now re-write Theorem \ref{ProdMain} as follows.

\begin{corollary} \label{ReProd}
For any $k, n \geq 0$, we have
$$\tH_k \tH_n = \sum_{\lambda \in \cP \atop{|\lambda| = k+n}} (-1)^{|\tlambda|} \theta_{\lambda, k, n}(q) R_{\lambda}.$$
\end{corollary}

The next result puts Corollary \ref{ReProd} in a satisfying algebraic context.

\begin{theorem} \label{basis} The set $\mathcal{R} = \{R_{\lambda} \, \mid \, \lambda \in \cP \}$ is a $\ZZ[q]$-basis for $\Lambda[q]$, where $R_{\lambda} = \tH_{m_1(\lambda)} e_{\tlambda}$.
\end{theorem}
\begin{proof}  We first show that $\ZZ[q]\text{-span}(\mathcal{R}) = \Lambda[q]$.  From the fact that the elementary symmetric functions form a basis for $\Lambda[q]$, and the fact that $R_{\lambda} = e_{\lambda}$ whenever $m_1(\lambda) = 0$, we only need to show that, for $n > 0$ and any $\nu \in \cP$ with $m_1(\nu) = 0$, $e_1^n e_{\nu} \in \ZZ[q]\text{-span}(\mathcal{R})$.  When $n=1$, since $e_1 = \tH_1$, then $e_1 e_{\nu} = \tH_1 e_{\nu}$, and there is nothing to prove.  Supposing the statement holds for $n$, we consider $e_1^{n+1} e_{\nu} = e_1 e_1^n e_{\nu}$, and suppose we have
$$ e_1^n e_{\nu} = \sum_{\lambda \in \cP} a_{\lambda}(q) R_{\lambda},$$
where a finite number of the coefficients $a_{\lambda}(q) \in \ZZ[q]$ are nonzero, and $|\lambda| = |\nu| + n$ for each such $\lambda$.  Then we have
\begin{align*}
e_1^{n+1} e_{\nu} & = e_1 \sum_{\lambda \in \cP} a_{\lambda}(q) R_{\lambda} = \tH_1 \sum_{\lambda \in \cP} a_{\lambda}(q) \tH_{m_1(\lambda)} e_{\tlambda} \\
 & = \sum_{\lambda \in \cP} a_{\lambda}(q) \tH_1 \tH_{m_1(\lambda)} e_{\tlambda}.
\end{align*}
Now, from Corollary \ref{ReProd}, $\tH_1 \tH_{m_1(\lambda)} \in \ZZ[q]$-span$\{R_{\eta} \, \mid \, |\eta| = 1 + m_1(\lambda) \}$, and for any $\lambda \in \cP$, $R_{\eta} e_{\tlambda} = R_{\mu}$, where $|\mu| = |\eta| + |\tlambda|$.  It follows that $e_1^{n+1} e_{\nu} \in \ZZ[q]\text{-span}(\mathcal{R})$, so that $\ZZ[q]\text{-span}(\mathcal{R}) = \Lambda[q]$.

For linear independence, it is enough to show that for each $k$, the set $\mathcal{R}^k = \{ R_{\lambda} \, \mid \, |\lambda| = k \}$ is linearly independent over $\ZZ[q]$.  Suppose that $b_{\lambda}(q) \in \ZZ[q]$, $|\lambda|=k$ satisfy
\begin{equation} \label{lincomb}
\sum_{|\lambda| = k} b_{\lambda}(q) R_{\lambda} = 0.
\end{equation}
We prove by reverse induction on $m_1(\lambda)$ that each $b_{\lambda}(q) = 0$.  If $m_1(\lambda) = k$, then $\lambda = (1^k)$, and $R_{\lambda} = \tH_k$.  Write each $R_{\lambda}$ in (\ref{lincomb}) as a $\ZZ[q]$-linear combination of the elementary symmetric functions.  Then by Lemma \ref{HeCoeff}, the coefficient of $e_{(1^k)}$ in the expansion of $R_{(1^k)} = \tH_k$ is $1$, while $e_{(1^k)}$ cannot appear in the expansion of any other $R_{\lambda}$ in (\ref{lincomb}).  Since the coefficient of $e_{(1^k)}$ in the expansion of (\ref{lincomb}) is $b_{(1^k)}(q)$, and the elementary symmetric functions are linearly independent, $b_{(1^k)}(q) = 0$.  Now let $j < k$ assume that $b_{\lambda}(q) = 0$ whenever $m_1(\lambda) >j$.  We must show that $b_{\lambda}(q) = 0$ whenever $m_1(\lambda) = j$.  We may prove this by reverse induction on $\tlambda$ with the lexicographical ordering.  The first case is $\tlambda = (k-j)$.  As before, expand (\ref{lincomb}) in the elementary symmetric function basis, and the coefficient of $e_{(1^j)} e_{(k-j)}$ must be $b_{(1^j (k-j))}(q)$, which then must be $0$.  Then, if $m_1(\mu)= j$, and $b_{\lambda}(q) = 0$ whenever $m_1(\lambda) = j$ and $\tlambda$ is greater than $\tmu$ in the lexicographical ordering, we may expand again in the elementary symmetric function basis, and use the induction hypothesis to see that $b_{\mu}(q) = 0$.  This completes the proof.
\end{proof}

We may immediately conclude the following property of the polynomials $\theta_{\lambda, n, k}(q)$.

\begin{corollary} \label{thetasym}
For any $\lambda$ (or $\vec{m}$), and any $n, k \geq 0$, we have $\theta_{\lambda, n, k} = \theta_{\lambda, k,n}$ (or $\theta_{\vec{m}, n, k} = \theta_{\vec{m}, k, n}$).
\end{corollary}
\begin{proof} We may apply Corollary \ref{ReProd} (or Theorem \ref{ProdMain}) to expand both $\tH_n \tH_k$ and $\tH_k \tH_n$ in terms of the $R_{\lambda}$.  By Theorem \ref{basis}, for any $\lambda$, the coefficient of $R_{\lambda}$ must be the same in each.
\end{proof}

For any $R_{\kappa}, R_{\nu} \in \mathcal{R}$, it follows from Theorem \ref{basis} that $R_{\kappa}R_{\nu}$ may be written uniquely as a $\ZZ[q]$-linear combination of elements in $\mathcal{R}$, or more precisely, if $|\kappa| + |\nu| = j$, elements in $\mathcal{R}^j = \{ R_{\gamma} \, \mid \, |\gamma| = j \}$.  That is, there are unique $\Theta_{\kappa, \nu, \gamma}(q) \in \ZZ[q]$ such that $R_{\kappa} R_{\nu} = \sum_{\gamma \in \cP} \Theta_{\kappa, \nu, \gamma}(q) R_{\gamma}$, where $\Theta_{\kappa, \nu, \gamma}(q)$ are called the \emph{structure constants} of the graded algebra $\Lambda[q]$ with respect to the $\ZZ[q]$-linear basis $\mathcal{R}$.  Corollary \ref{ReProd} may be applied in this situation as follows.  Let $k = m_1(\kappa)$, $n = m_1(\nu)$, so 
$$ R_{\kappa} R_{\nu} = e_{\tkappa} e_{\tnu} \tH_k \tH_n = e_{\tkappa} e_{\tnu} \sum_{\lambda \in \cP \atop{|\lambda| = k+n}} (-1)^{|\tlambda|} \theta_{\lambda, k, n}(q) R_{\lambda} = \sum_{\lambda \in \cP \atop{|\lambda| = k+n}} (-1)^{|\tlambda|} \theta_{\lambda, k, n}(q) R_{\lambda \cup \tkappa \cup \tnu},$$
where if $\alpha, \beta \in \cP$, then $\alpha \cup \beta$ is the partition obtained by taking the union of the multiset of their parts.

That is, if we write $R_{\kappa} R_{\nu} = \sum_{\gamma \in \cP} \Theta_{\kappa, \nu, \gamma}(q) R_{\gamma}$, then the structure constant $\Theta_{\kappa, \nu, \gamma}(q) = (-1)^{|\tlambda|} \theta_{\lambda, k, n}(q)$ whenever $\gamma = \lambda \cup \tkappa \cup \tnu$ for some $\lambda$ a partition of $k+n = m_1(\kappa) + m_1(\nu)$, and $\Theta_{\kappa, \nu, \gamma}(q) = 0$ otherwise.  So, the polynomials $\theta_{\lambda, k, n}(q)$ are, up to a sign, exactly these structure constants.\\
\\
\noindent{\bf Remark. } If one takes $t_1 = \cdots = t_l = 1$, then 
$$\tH_n(1, \ldots, 1) = \sum_{r_1 + \cdots + r_l =n} \binom{n}{r_1, \ldots, r_l}_q,$$
is the \emph{generalized Galois number}, which we denote by $G_n^{(l)}(q)$.  When $q$ is the power of a prime, it is known that $G_n^{(l)}(q)$ is the number of flags of length $l-1$ in an $n$-dimensional vector space over a field with $q$ elements (see \cite{BlKo13}, for example).  Then, the product $G_k^{(l)}(q) G_n^{(l)}(q)$ is the number of ordered pairs of such flags, the first from a $k$-dimensional space, and the second from an $n$-dimensional space.  Making the substitution $t_1 = \cdots = t_l = 1$ into the product formula in Theorem \ref{ProdMain} gives a curious alternating sum for this quantity, which may have some bijective proof through an inclusion-exclusion argument.  While we were unable to find such an argument, one would provide some enumerative meaning to the polynomials $\theta_{\lambda, n, k}(q)$ (or $\theta_{\vec{m}, n, k}(q)$), which would be a nice direction for future work.

\bigskip

\noindent
\begin{tabular}{ll}
\textsc{Department of Mathematics}\\ 
\textsc{College of William and Mary}\\
\textsc{P. O. Box 8795}\\
\textsc{Williamsburg, VA  23187}\\
{\em e-mail}:  {\tt spcameron@email.wm.edu}, {\tt vinroot@math.wm.edu}\\
\end{tabular}


\begin{thebibliography}{10}

\bibitem{An76}
G.~Andrews, The Theory of Partitions, Encyclopedia of Mathematics and its Applications, Addison-Wesley, Reading, Mass.-London-Amsterdam, 1976.

\bibitem{AsWi85}
R.~Askey and J.~Wilson, Some basic hypergeometric orthogonal polynomials that generalize Jacobi polynomials, \emph{Mem. Amer. Math. Soc.} \textbf{54}, no. 319, 1985. 

\bibitem{BlKo13}
T.~Bliem and S.~Kousidis, The number of flags in finite vector spaces: asymptotic normality and Mahonian statistics, \emph{J. Algebraic Combin.} \textbf{37} (2013), no. 2, 361--380.


\bibitem{Fi88}
N. J.~Fine, Basic Hypergeometric Series and Applications, Mathematical Surveys and Monographs 27, American Mathematical Society, Providence, RI, 1988.


\bibitem{Hi95}
K.~Hikami, Representations of motifs:  new aspect of the Rogers-Szeg\"o polynomials, \emph{J. Phys. Soc. Japan} \textbf{64} (1995), no. 4, 1047--1050.

\bibitem{Hi97}
K.~Hikami, Representation of the Yangian invariant motif and the Macdonald polynomial, \emph{J. Phys. A} \textbf{30} (1997), no. 7, 2447--2456.


\bibitem{Ma95}
I. G.~Macdonald, Symmetric functions and Hall polynomials, second edition, with contributions by A. Zelevinsky, Oxford Mathematical Monographs, Oxford Science Publications, The Clarendon Press, Oxford University Press, New York, 1995. 


\bibitem{Rog1}
L. J.~Rogers, On a three-fold symmetry in the elements of Heine's series, \emph{Proc. London Math. Soc.}, \textbf{24} (1893), 171--179.

\bibitem{Rog2}
L. J.~Rogers, On the expansion of some infinite products, \emph{Proc. London Math. Soc.}, \textbf{24} (1893), 337--352.



\bibitem{Sz26}
G.~Szeg\"o, Ein Beitrag zur Theorie der Thetafunktionen, \emph{S. B. Preuss. Akad. Wiss. Phys.-Math. Kl.}, (1926), 242--252.



\end{thebibliography}
\end{document}